\documentclass{amsart}
\usepackage{amsmath}
\usepackage{amssymb}
\usepackage{amsfonts}

\newtheorem{theorem}{Theorem}
\theoremstyle{plain}

\newtheorem{definition}{Definition}
\newtheorem{example}{Example}

\newtheorem{proposition}{Proposition}
\newtheorem{remark}{Remark}

\numberwithin{equation}{section}
\input{tcilatex}

\begin{document}
\title{On topological structures of fuzzy parametrized soft sets}
\author{S. ATMACA}
\address{Cumhuriyet University, Department of Mathematics, 58140 Sivas /
Turkey}
\email{seatmaca@cumhuriyet.edu.tr}
\author{\.{I}. ZORLUTUNA}
\address{Cumhuriyet University, Department of Mathematics, 58140 Sivas /
Turkey}
\email{ izorlu@cumhuriyet.edu.tr}

\begin{abstract}
In this paper, we introduce the topological structure of fuzzy parametrized
soft sets and fuzzy parametrized soft mappings. We define the notion of
quasi-coincidence for fuzzy parametrized soft sets and investigated basic
properties of it. We study the closure, interior, base, continuity and
compactness and properties of these concepts in fuzzy parametrized soft
topological spaces.
\end{abstract}

\maketitle

In 1965, after Zadeh \cite{Z} generalized the usual notion of a set with the
introduction of fuzzy set, the fuzzy set was carried out in the areas of
general theories and applied to many real life problems in uncertain,
ambiguous environment. In this manner, in 1968, Chang \cite{ch} gave the
definition of fuzzy topology and introduced the many topological notions in
fuzzy setting.

In 1999, Molodtsov \cite{M} introduced the concept of soft set theory which
is a comletely new approach for modelling uncertainty and pointed out
several directions for the applications of soft sets, such as game theory,
perron integrations, smoothness of functions and so on. To improve this
concept, many researchers applied this concept on topological spaces (e.g. 
\cite{ac1, chen, chen1, geor, sn, zo}), group theory, ring theory (e.g. \cite%
{Ak, AC, yc, F, yb}), and also decision making problems (e.g. \cite{CE, ce1,
dc, MB2}).

Recently, researchers have combined fuzzy set and soft set to generalize the
spaces and to solve more complicated problems. By this way, many interesting
applications of soft set theory have been expanded. First combination of
fuzzy set and soft set is fuzzy soft set and it was given by Maji and et al 
\cite{MB3}. Then fuzzy soft set theory has been applied in several
directions, such as topology (e.g. \cite{Az, Bh, TK, Ts}), various algebraic
structures (e.g. \cite{AA, io}) and especially decision making (e.g. \cite%
{fj, zk, arr, x2}). Another combination of fuzzy set and soft set was given
by \c{C}a\u{g}man and et al. \cite{ce2} and called it as fuzzy parametrized
soft set (as shortly FP-soft set). In that paper, \c{C}a\u{g}man and et al.
defined operations on FP-soft sets and improved several results. After that, 
\c{C}a\u{g}man and Deli \cite{Cd, Cd1} applied FP-soft sets to define some
decision making methods and applied these methods to problems that contain
uncertainties and fuzzy object.

In the present paper, we consider the topological structure of FP-soft sets.
Firstly, we give some basic ideas of FP-soft sets and also studied results.
We define FP-soft quasi-coincidence, as a generalization of
quasi-coincidence in fuzzy manner \cite{PM} and use this notion to
characterize concepts of FP-soft closure and FP-soft base in FP-soft
topological spaces. We also introduce the notion of mapping on FP-soft
classes and investigate the properties of FP-soft images and FP-soft inverse
images of FP-soft sets. We define FP-soft topology in Chang's sense. We
study the FP-soft closure and FP-soft interior operators and properties of
these concepts. Lastly we define FP-soft continuous mappings and we show
that image of a FP-soft compact space is also FP-soft compact.

\section{Preliminaries}

Throughout this paper $X$ denotes initial universe, $E$ denotes the set of
all possible parameters which are attributes, characteristic or properties
of the objects in $X$, and the set of all subsets of $X$ will be denoted by $%
P(X)$.

\begin{definition}
\cite{Z} A fuzzy set $A$ in $X$ is a function defined as follows:

$A=\{%
\mu
_{A}(x)/x:x\in X\}$\newline
where $%
\mu
_{A}:X\rightarrow \lbrack 0,1].$

Here $%
\mu
_{A}$ is called the membership function of $A$, and the value $%
\mu
_{A}(x)$ is called the grade of membership of $x\in X$. This value
represents the degree of $x$ belonging to the fuzzy set $A$.

A fuzzy point in $X$, whose value is $\alpha $ $(0<\alpha \leq 1)$ at the
support $x\in X,$ is denoted by $x_{\alpha }$. A fuzzy point $x_{\alpha }\in
A$, where $A$ is fuzzy set in $X$ iff $\alpha \leq \mu _{A}(x)$.
\end{definition}

\begin{definition}
\cite{M} A pair $(F,E)$ is called a soft set over $X$ if $F$ is a mapping
defined by $F:E\rightarrow P(X)$.

In the other words, a soft set is a parametrized family of subsets of the
set $X$. Each set $F(e)$, $e\in E$, from this family may be considered as
the set of $e$-elements of the soft set $(F,E)$.
\end{definition}

\begin{definition}
\cite{ce2} Let $A$ be a fuzzy set over $E$. An FP-soft set $F_{A}$ on the
universe $X$ is defined as follows:

$F_{A}=\{(%
\mu
_{A}(e)/e,f_{A}(e)):e\in E,f_{A}(e)\in P(X),%
\mu
_{A}(e)\in \lbrack 0,1]\}$,\newline
where the function $f_{A}:E\rightarrow P(X)$ is called approximate function
such that $f_{A}(e)=\varnothing $ if $%
\mu
_{A}(e)=0$, and the function $%
\mu
_{A}:E\rightarrow \lbrack 0,1]$ is called membership function of the set $A$.
\end{definition}

From now on, the set of all FP-soft sets over $X$\ will be denoted by $%
FPS(X,E)$\textbf{.}

\begin{definition}
\cite{ce2} Let $F_{A}\in FPS(X,E)$.

(1) $F_{A}$ is called the empty FP-soft set if $\mu _{A}(e)=0$ for all every 
$e\in E$, denoted by $F_{\varnothing }$.

(2) $F_{A}$ is called $A$-universal FP-soft set if $\mu _{A}(e)=1$ and $%
f_{A}(e)=X$ for all $e\in A$, denoted by $F\widetilde{_{A}}$.

If $A=E$, then $A$-universal FP-soft set is called universal FP-soft set,
denoted by $F\widetilde{_{E}}$.
\end{definition}

\begin{definition}
\cite{ce2} Let $F_{A},F_{B}\in FPS(X,E)$.

(1) $F_{A}$ is called a FP-soft subset of $F_{B}$ if $A\leq B$ and $%
f_{A}(e)\subseteq f_{B}(e)$ for every $e\in E$ and we write $F_{A}\widetilde{%
\subset }F_{B}$.

(2) $F_{A}$ and $F_{B}$ are said to be equal, denoted by $F_{A}=F_{B}$ if $%
F_{A}\widetilde{\subset }F_{B}$ and $F_{B}\widetilde{\subset }F_{A}$.

(3) The union of $F_{A}$ and $F_{B}$, denoted by $F_{A}\widetilde{\cup }%
F_{B} $, is the FP-soft set, defined by the membership and approximate
functions $\mu _{A\cup B}(e)=\max \{\mu _{A}(e),\mu _{B}(e)\}$ and $f_{A\cup
B}(e)=f_{A}(e)\cup g_{B}(e)$ for every $e\in E$, respectively.

(4) The intersection of $F_{A}$ and $F_{B}$, denoted by $F_{A}\widetilde{%
\cap }G_{B}$, is the FP-soft set, defined by the membership and approximate
functions $\mu _{A\cap B}(e)=\min \{\mu _{A}(e),\mu _{B}(e)\}$ and $f_{A\cap
B}(e)=f_{A}(e)\cap g_{B}(e)$ for every $e\in E$, respectively.
\end{definition}

\begin{definition}
\cite{ce2} Let $F_{A}\in FPS(X,E)$. Then the complement of $F_{A}$, denoted
by $F_{A}^{c}$, is the FP-soft set, defined by the membership and
approximate functions $\mu _{A^{c}}(e)=1-\mu _{A}(e)$ and $%
f_{A}^{c}(e)=X-f_{A}(e)$ for every $e\in E$, respectively.

Clearly $(F_{A}^{c})^{c}=F_{A}$, $F_{\widetilde{E}}^{c}=F_{\varnothing }$
and $F_{\varnothing }^{c}=F_{\widetilde{E}}$
\end{definition}

\begin{proposition}
\cite{ce2} Let $F_{A}$, $F_{B}$ and $F_{C}\in FPS(X,E)$. Then

(1)$\ (F_{A}\widetilde{\cup }F_{B})^{c}=F_{A}^{c}\widetilde{\cap }F_{B}^{c}$

(2)$\ (F_{A}\widetilde{\cap }F_{B})^{c}=F_{A}^{c}\widetilde{\cup }F_{B}^{c}$

(3) $F_{A}\widetilde{\cap }F_{A}=F_{A}$, $F_{A}\widetilde{\cup }F_{A}=F_{A}$

(4) $F_{A}\widetilde{\cap }F_{\varnothing }=F_{\varnothing }$, $F_{A}%
\widetilde{\cap }F_{\widetilde{E}}=F_{A}$

(5) $F_{A}\widetilde{\cap }F_{B}=F_{B}\widetilde{\cap }F_{A}$, $F_{A}%
\widetilde{\cup }F_{B}=F_{B}\widetilde{\cup }F_{A}$

(6) $F_{A}\widetilde{\cap }(F_{B}\widetilde{\cap }F_{C})=$ $(F_{A}\widetilde{%
\cap }F_{B})\widetilde{\cap }F_{C}$, $F_{A}\widetilde{\cup }(F_{B}\widetilde{%
\cup }F_{C})=$ $(F_{A}\widetilde{\cup }F_{B})\widetilde{\cup }F_{C}$

(7) $F_{A}\widetilde{\cup }F_{\varnothing }=F_{A}$, $F_{A}\widetilde{\cup }%
F_{\widetilde{E}}=F_{\widetilde{E}}$
\end{proposition}

\section{some properties of FP-soft sets and FP-soft mappings}

\begin{definition}
Let $J$ be an arbitrary index set and $F_{A_{i}}\in FPS(X,E)$ for all $i\in
J $.

(1) The union of $F_{A_{i}}$'s, denoted by $\widetilde{\underset{i\in J}{%
{\LARGE \cup }}}F_{A_{i}}$, is the FP-soft set, defined by the membership
and approximate functions $\mu _{\underset{i\in J}{{\LARGE \cup }}A_{i}}(e)=%
\underset{i\in J}{\sup }\{\mu _{A_{i}}(e)\}$ and $f_{\underset{i\in J}{%
{\LARGE \cup }}A_{i}}(e)=\underset{i\in J}{{\LARGE \cup }}f_{A_{i}}(e)$ for
every $e\in E$, respectively.

(2) The intersection of $F_{A_{i}}$'s, denoted by $\widetilde{\underset{i\in
J}{{\LARGE \cap }}}F_{A_{i}}$, is the FP-soft set, defined by the membership
and approximate functions $\mu _{\underset{i\in J}{{\LARGE \cap }}A_{i}}(e)=%
\underset{i\in J}{\inf }\{\mu _{A_{i}}(e)\}$ and $f_{\underset{i\in J}{%
{\LARGE \cap }}A_{i}}(e)=\underset{i\in J}{{\LARGE \cap }}f_{A_{i}}(e)$ for
every $e\in E$, respectively.
\end{definition}

\begin{proposition}
Let $J$ be an arbitrary index set and $F_{A_{i}}\in FPS(X,E)$ for all $i\in
J $. Then

(1)$\ (\widetilde{\underset{i\in J}{{\LARGE \cup }}}F_{A_{i}})^{c}=%
\widetilde{\underset{i\in J}{{\LARGE \cap }}}F_{A_{i}}^{c}$.

(2)$\ (\widetilde{\underset{i\in J}{{\LARGE \cap }}}F_{A_{i}})^{c}=%
\widetilde{\underset{i\in J}{{\LARGE \cup }}}F_{A_{i}}^{c}$.
\end{proposition}

\begin{proof}
(1) Put $F_{B}=(\widetilde{\underset{i\in J}{{\LARGE \cup }}}F_{A_{i}})^{c}$
and $F_{C}=\widetilde{\underset{i\in J}{{\LARGE \cap }}}F_{A_{i}}^{c}$. Then
for all $e\in E$,\newline
$\mu _{B}(e)=1-\mu _{\underset{i\in J}{{\LARGE \cup }}A_{i}}(e)=1-\underset{%
i\in J}{\sup }\{\mu _{A_{i}}(e)\}=\underset{i\in J}{\inf }\{1-\mu
_{A_{i}}(e)\}=\underset{i\in J}{\inf }\{\mu _{A_{i}^{c}}(e)\}=\mu _{C}(e)$

and \newline
$f_{B}(e)=X-f_{\underset{i\in J}{{\LARGE \cup }}A_{i}}(e)=X-\underset{i\in J}%
{{\LARGE \cup }}f_{A_{i}}(e)=\underset{i\in J}{{\LARGE \cap }}%
(X-f_{A_{i}}(e))=\underset{i\in J}{{\LARGE \cap }}f_{A_{i}^{c}}(e)=f_{%
\underset{i\in J}{{\LARGE \cap }}A_{i}^{c}}(e)=f_{C}(e)$.

This completes the proof. The other can be proved similarly
\end{proof}

\begin{definition}
The FP-soft set $F_{A}\in FPS(X,E)$ is called FP-soft point if $A$ is fuzzy
singleton and $f_{A}(e)\in P(X)$ for $e\in $supp$A$. If $A=\{e\}$, $\mu
_{A}(e)=\alpha \in (0,1]$, then we denote this FP-soft point by $e_{\alpha
}^{f}$.
\end{definition}

\begin{definition}
Let $e_{\alpha }^{f}$, $F_{A}\in FPS(X,E)$. We say that $e_{\alpha }^{f}%
\widetilde{\in }F_{A}$ read as $e_{\alpha }^{f}$ belongs to the FP-soft set $%
F_{A}$ if $\alpha \leq \mu _{A}(e)$ and $f(e)\subseteq f_{A}(e)$.
\end{definition}

\begin{proposition}
Every non empty FP-soft set $F_{A}$ can be expresssed as the union of all
the FP-soft points which belong to $F_{A}$.
\end{proposition}

\begin{proof}
This follows from the fact that any fuzzy set is the union of fuzzy points
which belong to it \cite{PM}.
\end{proof}

\begin{definition}
Let $F_{A}$, $F_{B}\in FPS(X,E)$. $F_{A}$ is said to be FP-soft
quasi-coincident with $F_{B}$, denoted by $F_{A}qF_{B}$, if there exists $%
e\in E$ such that $\mu _{A}(e)+\mu _{B}(e)>1$ or $f_{A}(e)$ is not subset of 
$f_{B}^{c}(e)$. If $F_{A}$ is not FP-soft quasi-coincident with $F_{B}$,
then we write $F_{A}\overline{q}F_{B}$.
\end{definition}

\begin{definition}
Let $e_{\alpha }^{f}$, $F_{A}\in FPS(X,E)$. $e_{\alpha }^{f}$ is said to be
FP-soft quasi-coincident with $F_{A}$, denoted by $e_{\alpha }^{f}qF_{A}$,
if $\alpha +\mu _{A}(e)>1$ or $f(e)$ is not subset of $f_{A}^{c}(e)$. If $%
e_{\alpha }^{f}$ is not FP-soft quasi-coincident with $F_{A}$, then we write 
$e_{\alpha }^{f}\overline{q}F_{A}$.
\end{definition}

\begin{proposition}
\label{P-c}Let $F_{A}$, $F_{B}\in FPS(X,E)$, Then the following are true.

(1) $F_{A}\widetilde{\subseteq }F_{B}\Leftrightarrow F_{A}\overline{q}%
F_{B}^{c}$.{}

(2) $F_{A}qF_{B}\Rightarrow F_{A}\widetilde{\cap }F_{B}\neq F_{\varnothing 
\text{.}}$

(3) $F_{A}\overline{q}F_{A}^{c}$.

(4) $F_{A}qF_{B}\Leftrightarrow $there exists an $e_{\alpha }^{f}\widetilde{%
\in }F_{A}$ such that $e_{\alpha }^{f}qF_{B}$.

(5) $e_{\alpha }^{f}\widetilde{\in }F_{A}^{c}\Leftrightarrow e_{\alpha }^{f}%
\overline{q}F_{A}$.

(6) $F_{A}\widetilde{\subseteq }F_{B}\Rightarrow $ If $e_{\alpha }^{f}qF_{A}$%
, then $e_{\alpha }^{f}qF_{B}$ for all $e_{\alpha }^{f}\in FPS(X,E)$.
\end{proposition}

\begin{proof}
(1)

$%
\begin{array}{ll}
F_{A}\widetilde{\subseteq }F_{B} & \Leftrightarrow \text{for all }e\in E%
\text{, }\mu _{A}(e)\leq \mu _{B}(e)\text{ and }f_{A}(e)\subseteq f_{B}(e)
\\ 
& \Leftrightarrow \text{for all }e\in E\text{ , }\mu _{A}(e)-\mu _{B}(e)\leq
0\text{ and }f_{A}(e)\subseteq f_{B}(e) \\ 
& \Leftrightarrow \text{for all }e\in E\text{, }\mu _{A}(e)+1-\mu
_{B}(e)\leq 1\text{ and }f_{A}(e)\subseteq f_{B}(e) \\ 
& \Leftrightarrow F_{A}\overline{q}F_{B}^{c}%
\end{array}%
$

(2) Let $F_{A}qF_{B}$. Then there exists an $e\in E$ such that $\mu
_{A}(e)+\mu _{B}(e)>1$ or $f_{A}(e)$ is not subset of $f_{B}^{c}(e)$. If $%
\mu _{A}(e)+\mu _{B}(e)>1$, $A\wedge B\neq 0_{E}$ and the proof is easy. If $%
f_{A}(e)$ is not subset of $f_{B}^{c}(e)$, then $f_{A}(e)\cap f_{B}(e)\neq
\varnothing $. Hence $F_{A}\widetilde{\cap }F_{B}\neq F_{\varnothing }$.

(3) Suppose that $F_{A}qF_{A}^{c}$. Then there exists $e\in E$ such that $%
\mu _{A}(e)+\mu _{A^{c}}(e)>1$ or $f_{A}(e)$ is not subset $%
(f_{A}^{c}(e))^{c}$. But this is impossible.

(4) If $F_{A}qF_{B}$, then there exist an $e\in E$\ such that $\mu
_{A}(e)+\mu _{B}(e)>1$\ or $f_{A}(e)$\ is not subset of $f_{B}^{c}(e)$. Put $%
\alpha =\mu _{A}(e)$ and $f(e)=f_{A}(e)$. Then we have $e_{\alpha }^{f}%
\widetilde{\in }F_{A}$\ and $e_{\alpha }^{f}qF_{B}$.

Conversely, suppose that $e_{\alpha }^{f}qF_{B}$ for some $e_{\alpha }^{f}%
\widetilde{\in }F_{A}$. Then $\alpha +\mu _{B}(e)>1$ or $f(e)$ is not subset
of $f_{B}^{c}(e)$. Therefore, we have $\mu _{A}(e)+\mu _{B}(e)>1$ or $%
f_{A}(e)$ is not subset of $f_{B}^{c}(e)$ for $e\in E$ . This shows that $%
F_{A}qF_{B}$.

(5) It is obvious from (1).

(6) Let $e_{\alpha }^{f}$, $F_{A}\in FPS(X,E)$ and $e_{\alpha }^{f}qF_{A}$.
Then $\alpha +\mu _{A}(e)>1$ or $f(e)$ is not subset of $f_{A}^{c}(e)$.
Since $F_{A}\widetilde{\subseteq }F_{B}$, $\alpha +\mu _{B}(e)>1$ or $f(e)$
is not subset of $f_{B}^{c}(e)$. Hence we have $e_{\alpha }^{f}qF_{B}$.
\end{proof}

\begin{proposition}
Let $\{F_{A_{i}}:i\in J\}$ be a family of FP-soft sets in $FPS(X,E)$ where $%
J $ is an index set. Then $e_{\alpha }^{f}$ is FP-soft quasi-coincident with 
$\widetilde{{\LARGE \cup }}_{i\in J}F_{A_{i}}$ if and only if there exists
some $F_{A_{i}}\in \{F_{A_{i}}:i\in J\}$ such that $e_{\alpha
}^{f}qF_{A_{i}} $.
\end{proposition}

\begin{proof}
Obvious.
\end{proof}

\begin{definition}
Let $FPS(X,E)$ and $FPS(Y,K)$ be families of all FP-soft sets over $X$ and $%
Y $, respectively. Let $u:X\rightarrow Y$ and $p:E\rightarrow K$ be two
functions. Then a FP-soft mapping $f_{up}:FPS(X,E)\rightarrow FPS(Y,K)$ is
defined as:

(1) for $F_{A}\in FPS(X,E)$, the image of $F_{A}$ under the FP-soft mapping $%
f_{up}$ is the FP-soft set $G_{S}$ over $Y$ defined by the approximate
function, $\forall k\in K$,

$g_{S}(k)=\left\{ 
\begin{array}{cc}
\underset{e\in p^{-1}(k)}{{\Large \cup }}u(f_{A}(e))\text{,} & \text{if }%
p^{-1}(k)\neq \varnothing \text{;} \\ 
\varnothing \text{,} & \text{otherwise.}%
\end{array}%
\right. $

where $p(A)=S$ is fuzzy set in $K$.

(2) for $G_{S}\in FPS(Y,K)$, then the pre-image of $G_{S}$ under the FP-soft
mapping $f_{up}$ is the FP-soft set $F_{A}$ over $X$ defined by the
approximate function, $\forall e\in E$

$f_{A}(e)=u^{-1}(g_{S}(p(e)))$ where $p^{-1}(S)=A$ is fuzzy set in $E$.
\end{definition}

If $u$ and $p$ is injective, then the FP-soft mapping $f_{up}$ is said to be
injective. If $u$ and $p$ is surjective, then the FP-soft mapping $f_{up}$
is said to be surjective. The FP-soft mapping $f_{up}$ is called constant,
if $u$ and $p$ are constant.

\begin{theorem}
\label{fo}Let $X$ and $Y$ crips sets $F_{A}$, $F_{A_{i}}\in FPS(X,E)$, $%
G_{S} $, $G_{S_{i}}\in FPS(Y,K)$ $\forall i\in J$, where $J$ is an index
set. Let $f_{up}:FPS(X,E)\rightarrow FPS(Y,K)$ be a FP-soft mapping. Then,

(1) If $F_{A_{1}}\widetilde{\subset }F_{A_{2}}$ then $f_{up}(F_{A_{1}})%
\widetilde{\subset }f_{up}(F_{A_{2}})$.

(2) If $G_{S_{1}}\widetilde{\subset }G_{S_{2}}$ then $f_{up}^{-1}(G_{S_{1}})%
\widetilde{\subset }f_{up}^{-1}(G_{S_{2}})$.

(3) $F_{A}\widetilde{\subset }f_{up}^{-1}(f_{up}(F_{A}))$, the equality
holds if $f_{up}$ is injective.

(4) $f_{up}(f_{up}^{-1}(G_{S}))\widetilde{\subset }G_{S}$, the equality
holds if $f_{up}$ is surjective.

(5) $f_{up}(\widetilde{{\Large \cup }}_{i\in J}F_{A_{i}})=\widetilde{{\Large %
\cup }}_{i\in J}f_{up}(F_{A_{i}})$.

(6) $f_{up}(\widetilde{{\Large \cap }}_{i\in J}F_{A_{i}})\widetilde{\subset }%
\widetilde{{\Large \cap }}_{i\in J}f_{up}(F_{A_{i}})$, the equality holds if 
$f_{up}$ is injective.

(7) $f_{up}^{-1}(\widetilde{{\Large \cup }}_{i\in J}G_{S_{i}})=\widetilde{%
{\Large \cup }}_{i\in J}f_{up}^{-1}(G_{S_{i}})$.

(8) $f_{up}^{-1}(\widetilde{{\Large \cap }}_{i\in J}G_{S_{i}})=\widetilde{%
{\Large \cap }}_{i\in J}f_{up}^{-1}(G_{S_{i}})$.

(9) $(f_{up}^{-1}(G_{S}))^{c}=f_{up}^{-1}(G_{S}^{c})$.

(10) $(f_{up}(F_{A}))^{c}\widetilde{\subset }f_{up}(F_{A}^{c})$.

(11) $f_{up}^{-1}(G_{\widetilde{K}})=F_{\widetilde{E}}$.

(12) $f_{up}^{-1}(G_{\varnothing })=F_{\varnothing }$.

(13) $f_{up}(F_{\widetilde{E}})\widetilde{\subset }G_{\widetilde{K}}$, the
equality holds if $f_{up}$ is surjective.

(14) $f_{up}(F_{\varnothing })=G_{\varnothing }$.
\end{theorem}

\begin{proof}
We only prove (3),(5),(7),(9),(11) and (12). The others can be proved
similarly.

(3) Put $G_{S}=f_{up}(F_{A})$ and $F_{B}=$ $f_{up}^{-1}(G_{S})$. Since $%
A\leq p^{-1}(p(A))=p^{-1}(S)=B$, It is sufficient to show $f_{A}(e)\subseteq
f_{B}(e)$ for all $e\in E$,

$%
\begin{array}{ll}
f_{B}(e) & =u^{-1}(g_{S}(p(e))) \\ 
& =u^{-1}({\Large \cup }_{e\in p^{-1}(p(e))}u(f_{A}(e))) \\ 
& ={\Large \cup }_{e\in p^{-1}(p(e))}u^{-1}(u(f_{A}(e))) \\ 
& \supseteq f_{A}(e)%
\end{array}%
$

This completes the proof.

(5) Put $G_{S_{i}}=f_{up}(F_{A_{i}})$ and $G_{S}=f_{up}(\widetilde{{\Large %
\cup }}_{i\in J}(F_{A_{i}}))$. Then $S=p({\Large \vee }A_{i})={\Large \vee }%
p(A_{i})={\Large \vee }S_{i}$ and for all $k\in K$,

$%
\begin{array}{ll}
g_{S}(k) & =\left\{ 
\begin{array}{cc}
{\Large \cup }_{e\in p^{-1}(k)}u(\underset{i\in J}{\cup }f_{A_{i}}(e)) & 
\text{if }p^{-1}(k)\neq \varnothing \text{;} \\ 
\varnothing & \text{otherwise.}%
\end{array}%
\right. \\ 
& =\left\{ 
\begin{array}{cc}
{\Large \cup }_{e\in p^{-1}(k)}\underset{i\in J}{\cup }u(f_{A_{i}}(e)) & 
\text{if }p^{-1}(k)\neq \varnothing \text{;} \\ 
\varnothing & \text{otherwise.}%
\end{array}%
\right. \\ 
& =\left\{ 
\begin{array}{cc}
\underset{i\in J}{{\LARGE \cup }}{\Large \cup }_{e\in
p^{-1}(k)}u(f_{A_{i}}(e)) & \text{if }p^{-1}(k)\neq \varnothing \text{;} \\ 
\varnothing & \text{otherwise.}%
\end{array}%
\right. \\ 
& =\underset{i\in J}{{\LARGE \cup }}\left\{ 
\begin{array}{cc}
{\Large \cup }_{e\in p^{-1}(k)}u(f_{A_{i}}(e)) & \text{if }p^{-1}(k)\neq
\varnothing \text{;} \\ 
\varnothing & \text{otherwise.}%
\end{array}%
\right. \\ 
& ={\Large \cup }_{i\in J}g_{S_{i}}(k)%
\end{array}%
$

This completes the proof.

(7) Put $F_{A_{i}}=f_{up}^{-1}(G_{S_{i}})$ and $F_{A}=f_{up}^{-1}(\widetilde{%
{\Large \cup }}_{i\in J}G_{S_{i}})$. Then $A=p^{-1}({\Large \vee }S_{i})=%
{\Large \vee }p^{-1}(S_{i})={\Large \vee }A_{i}$ and for all $e\in E$,

$%
\begin{array}{ll}
f_{A}(e) & =u^{-1}(\underset{i\in I}{{\LARGE \cup }}g_{S_{i}}(p(e))) \\ 
& =\underset{i\in I}{{\LARGE \cup }}u^{-1}(g_{S_{i}}(p(e))) \\ 
& ={\Large \cup }_{i\in J}f_{A_{i}}(e)%
\end{array}%
$

This completes the proof.

(9) Put $f_{up}^{-1}(G_{S})=F_{A}$ and $f_{up}^{-1}(G_{S}^{c})=F_{B}$. Then
for all $e\in E$,\newline
$f_{B}(e)=f_{p^{-1}(S^{c})}(e)=f_{(p^{-1}(S))^{c}}(e)=f_{A^{c}}(e)$ where $%
p^{-1}(S)$ and $p^{-1}(S^{c})$ are fuzzy sets over $E$. This shows that the
approximate functions of $F_{B}$ and $F_{A}^{c}$ are equal. This completes
the proof.

(11) Put $F_{A}=f_{up}^{-1}(G_{\widetilde{K}})$. Then for all $e\in E$,%
\newline
$f_{A}(e)=u^{-1}(G_{\widetilde{K}}(p(e)))=u^{-1}(Y)=X=f_{E}(e)$. This shows
that $F_{A}=F_{\widetilde{E}}$.

(12) Since $p^{-1}(K)$ is fuzzy empty set i.e. $0_{E}$, the proof is clear.
\end{proof}

\section{FP-soft topological spaces}

\begin{definition}
A FP-soft topological space is a pair $(X,{\Large \tau })$ where $X$ is a
nonempty set and ${\Large \tau }$ is a family of FP-soft sets over $X$
satisfying the following properties:

(T1) $F_{\varnothing },F_{\widetilde{E}}\in {\Large \tau }$

(T2) If $F_{A}$, $F_{B}\in {\Large \tau }$ , then $F_{A}\widetilde{\cap }%
F_{B}\in {\Large \tau }$

(T3) If $F_{A_{i}}\in {\Large \tau }$ ,$\forall i\in J$, then $\widetilde{%
{\LARGE \cup }}_{i\in J}F_{A_{i}}\in {\Large \tau }$.\newline
${\Large \tau }$ is called a topology of FP-soft sets on $X$. Every member
of ${\Large \tau }$ is called FP-soft open in $(X,{\Large \tau })$. $F_{B}$
is called FP-soft closed in $(X,{\Large \tau })$ if $F_{B}^{c}\in {\Large %
\tau }$.
\end{definition}

\begin{example}
${\Large \tau }_{indiscrete}=\{F_{\varnothing },F_{\widetilde{E}}\}$ is a
FP-soft topology on $X$.

${\Large \tau }_{discrete}=FPS(X,E)$ is a FP-soft topology on $X$.
\end{example}

\begin{example}
\label{to}Assume that $X=\{x_{1},x_{2},x_{3},x_{4}\}$ is a universal set and 
$E=\{e_{1},e_{2},e_{3}\}$ is a set of parameters. If

$F_{A_{1}}=\{((e_{1})_{0,2},\{x_{1},x_{3}\}),((e_{2})_{0,3},\{x_{1},x_{4}%
\}),((e_{3})_{0,4},\{x_{2}\})\}$

$F_{A_{2}}=\{((e_{1})_{0,2},\{x_{1},x_{2},x_{3}\}),((e_{2})_{0,5},%
\{x_{1},x_{4}\}),((e_{3})_{0,4},\{x_{1},x_{2}\})\}$

$F_{A_{3}}=\{((e_{1})_{0,7},\{x_{1},x_{3}%
\}),((e_{2})_{0,3},X),((e_{3})_{0,9},\{x_{2},x_{3}\})\}$

$F_{A_{4}}=\{((e_{1})_{0,7},\{x_{1},x_{2},x_{3}%
\}),((e_{2})_{0,5},X),((e_{3})_{0,9},\{x_{2},x_{3}\})\}$\newline
then ${\Large \tau }=\{F_{\varnothing
},F_{A_{1}},F_{A_{2}},F_{A_{3}},F_{A_{4}},F_{\widetilde{E}}\}$ is a FP-soft
topology on $X$.
\end{example}

\begin{theorem}
Let $(X,{\Large \tau })$ a FP-soft topological space and ${\Large \tau }%
^{\prime }$ denote family of all closed sets. Then;

(1) $F_{\varnothing },F_{\widetilde{E}}\in {\Large \tau }^{\prime }$

(2) If $F_{A}$, $F_{B}\in {\Large \tau }^{\prime }$, then $F_{A}\widetilde{%
\cup }F_{B}\in {\Large \tau }^{\prime }$

(3) If $F_{A_{i}}\in {\Large \tau }^{\prime }$, $\forall i\in J$, then $%
\widetilde{{\LARGE \cap }}_{i\in J}F_{A_{i}}\in {\Large \tau }^{\prime }$.
\end{theorem}

\begin{proof}
Straightforward.
\end{proof}

\begin{definition}
Let $(X,{\Large \tau })$ be a FP-soft topological space and $F_{A}\in
FPS(X,E)$. The FP-soft closure of $F_{A}$ in $(X,{\Large \tau })$, denoted
by $\overline{F_{A}}$, is the intersection of all FP-soft closed supersets
of $F_{A}.$

Clearly, $\overline{F_{A}}$ is the smallest FP-soft closed set over $X$
which contains $F_{A}$, and $\overline{F_{A}}$ is closed.
\end{definition}

\begin{theorem}
\label{kap-oz}Let $(X,{\Large \tau })$ be a FP-soft topological space and $%
F_{A}$, $F_{B}\in FPS(X,E)$. Then,

(1) $\overline{F_{\varnothing }}=F_{\varnothing }$ and $\overline{F_{%
\widetilde{E}}}=F_{\widetilde{E}}$.

(2) $F_{A}\widetilde{\subset }\overline{F_{A}}$.

(3) $\overline{\overline{F_{A}}}=\overline{F_{A}}$.

(4) If $F_{A}\widetilde{\subset }F_{B}$, then $\overline{F_{A}}\widetilde{%
\subset }\overline{F_{B}}$.

(5) $F_{A}$ is a FP-soft closed set if and only if $F_{A}=\overline{F_{A}}$.

(6) $\overline{F_{A}\widetilde{\cup }F_{B}}=\overline{F_{A}}\widetilde{\cup }%
\overline{F_{B}}$.
\end{theorem}

\begin{proof}
(1),(2),(3) and (4) are obvious from the definition of FP-soft closure.

(5) Let $F_{A}$ be a FP-soft closed set. Since\ $\overline{F_{A}}$ is the
smallest FP-soft closed set which contains $F_{A}$, then $\overline{F_{A}}%
\widetilde{\subset }F_{A}$. Therefore, $F_{A}=\overline{F_{A}}$.

(6) Since $F_{A}\widetilde{\subset }F_{A}\widetilde{\cup }F_{B}$ and $F_{B}%
\widetilde{\subset }F_{A}\widetilde{\cup }F_{B}$, then, by (4), $\overline{%
F_{A}}\widetilde{\subset }\overline{F_{A}\widetilde{\cup }F_{B}}$ , $%
\overline{F_{B}}\widetilde{\subset }\overline{F_{A}\widetilde{\cup }F_{B}}$
and hence $\overline{F_{A}}\widetilde{\cup }\overline{F_{B}}\widetilde{%
\subset }\overline{F_{A}\widetilde{\cup }F_{B}}$.

Conversely, since $\overline{F_{A}}$, $\overline{F_{B}}$ are FP-soft closed
sets, $\overline{F_{A}}\widetilde{\cup }\overline{F_{B}}$ is a FP-soft
closed set. Again since $F_{A}\widetilde{\cup }F_{B}\widetilde{\subset }%
\overline{F_{A}}\widetilde{\cup }\overline{F_{B}}$, by (4), then $\overline{%
F_{A}\widetilde{\cup }F_{B}}\widetilde{\subset }\overline{F_{A}}\widetilde{%
\cup }\overline{F_{B}}$.
\end{proof}

\begin{definition}
Let $(X,{\Large \tau })$ be a FP-soft topological space. A FP-soft set $%
F_{A} $ in $FPS(X,E)$ is called FP-Q-neighborhood (briefly, FP-Q-nbd) of a
FP-soft set $F_{B}$ if there exists a FP-soft open set $F_{C}$ in ${\Large %
\tau }$ such that $F_{B}qF_{C}$ and $F_{C}\widetilde{\subseteq }F_{A}$.
\end{definition}

\begin{theorem}
Let $e_{\alpha }^{f}$, $F_{A}\in FPS(X,E)$. Then $e_{\alpha }^{f}\widetilde{%
\in }\overline{F_{A}}$ if and only if each FP-Q-nbd of $e_{\alpha }^{f}$ is
FP-soft quasi-coincident with $F_{A}$.
\end{theorem}

\begin{proof}
Let $e_{\alpha }^{f}\widetilde{\in }\overline{F_{A}}$. Suppose that $F_{C}$
is a FP-Q-nbd of $e_{\alpha }^{f}$ and $F_{C}\overline{q}F_{A}$. Then there
exists a FP-soft open set $F_{B}$ such that $e_{\alpha }^{f}qF_{B}\widetilde{%
\subseteq }F_{C}$. Since $F_{C}\overline{q}F_{A}$, by Proposition \ref{P-c}%
(1), $F_{A}\widetilde{\subseteq }F_{C}^{c}\widetilde{\subseteq }F_{B}^{c}$.
Again since $e_{\alpha }^{f}qF_{B}$, $e_{\alpha }^{f}$ does not belong to $%
F_{B}^{c}$. This is a contradiction with $\overline{F_{A}}\widetilde{%
\subseteq }F_{B}^{c}$.

Conversely, Let each Q-nbd of $e_{\alpha }^{f}$ be FP-soft quasi-coincident
with $F_{A}$. Suppose that $e_{\alpha }^{f}$ does not belong to $\overline{%
F_{A}}$. Then there exists a FP-soft closed set $F_{B}$ which is containing $%
F_{A}$ such that $e_{\alpha }^{f}$ does not belong to $F_{B}$. By
Proposition \ref{P-c}(5), we have $e_{\alpha }^{f}qF_{B}^{c}$. Then $%
F_{B}^{c}$ is a FP-Q-nbd of $e_{\alpha }^{f}$ and by Proposition \ref{P-c}%
(1), $F_{A}\overline{q}F_{B}^{c}$. This is a contradiction with the
hypothesis.
\end{proof}

\begin{definition}
Let $(X,{\Large \tau })$ be a FP-soft topological space and $F_{A}\in
FPS(X,E)$. The FP-soft interior of $F_{A}$ denoted by $F_{A}^{\circ }$ is
the union of all FP-soft open subsets of $F_{A}$.

Clearly, $F_{A}^{\circ }$ is the largest fuzzy soft open set contained in $%
F_{A}$ and $F_{A}^{\circ }$ is FP-soft open.
\end{definition}

\begin{theorem}
\label{ic-oz}Let $(X,{\Large \tau })$ be a FP-soft topological space and $%
F_{A}$, $F_{B}\in FPS(X,E)$. Then,

(1) $(F_{\varnothing })^{\circ }=F_{\varnothing }$ and $(F_{\widetilde{E}%
})^{\circ }=F_{\widetilde{E}}$.

(2) $F_{A}^{\circ }\widetilde{\subset }F_{A}$.

(3) $\left( F_{A}^{\circ }\right) ^{\circ }=F_{A}^{\circ }$.

(4) If $F_{A}\widetilde{\subset }F_{B}$, then $F_{A}^{\circ }\widetilde{%
\subset }F_{B}^{\circ }$.

(5) $F_{A}$ is a FP-soft open set if and only if $F_{A}=F_{A}^{\circ }$.

(6) $\left( F_{A}\widetilde{\cap }F_{B}\right) ^{\circ }=F_{A}^{\circ }%
\widetilde{\cap }F_{B}^{\circ }$.
\end{theorem}

\begin{proof}
Similar to that of Theorem \ref{kap-oz}.
\end{proof}

\begin{theorem}
\label{ik}Let $(X,{\Large \tau })$ be a FP-soft topological space and $%
F_{A}\in FPS(X,E)$. Then,

(1) $\left( F_{A}^{\circ }\right) ^{c}=\overline{F_{A}^{c}}$.

(2)$\left( \overline{F_{A}}\right) ^{c}=\left( F_{A}^{c}\right) ^{\circ }$.
\end{theorem}

\begin{proof}
We only prove (1). The other is similar.

$%
\begin{array}{ll}
\left( F_{A}^{\circ }\right) ^{c} & =(\widetilde{\cup }\{F_{B}\left\vert
F_{B}\in \tau \text{, }F_{A}\widetilde{\subset }F_{B}\text{ }\right\} )^{c}
\\ 
& =\widetilde{\cap }\left\{ F_{B}^{c}\right. \left\vert F_{B}\in \tau \text{%
, }F_{A}\widetilde{\subset }F_{B}\text{ }\right\} \\ 
& =\widetilde{\cap }\left\{ F_{B}^{c}\right. \left\vert F_{B}^{c}\in \tau
^{\prime }\text{, }F_{B}^{c}\widetilde{\subset }F_{A}^{c}\text{ }\right\} \\ 
& =\overline{F_{A}^{c}}%
\end{array}%
$
\end{proof}

\begin{theorem}
\label{kap-op}Let $c:FPS(X,E)\rightarrow FPS(X,E)$ be an operator satisfying
the following:

(c1) $c(F_{\varnothing })=F_{\varnothing }$.

(c2) $F_{A}\widetilde{\subseteq }c(F_{A}),\forall F_{A}\in FPS(X,E)$.

(c3) $c(F_{A}\widetilde{\cup }F_{B})=c(F_{A})\widetilde{\cup }%
c(F_{B}),\forall F_{A},F_{B}\in FPS(X,E)$.

(c4) $c(c(F_{A}))=c(F_{A}),\forall F_{A}\in FPS(X,E)$.\newline
Then we can associate FP-soft topology in the following way:

$\tau =\{F_{A}^{c}\in FPS(X,E)|c(F_{A})=F_{A}\}.$\newline
Moreover with this FP-soft topology $\tau $, $\overline{F_{A}}=c(F_{A})$ for
every $F_{A}\in FPS(X,E)$.
\end{theorem}

\begin{proof}
(T1) By (c1), $F_{\varnothing }^{c}=F_{\widetilde{E}}\in \tau $. By (c2) $F_{%
\widetilde{E}}\widetilde{\subseteq }c(F_{\widetilde{E}})$, so $c(F_{%
\widetilde{E}})=F_{\widetilde{E}}$ and $F_{\varnothing }\in \tau $.

(T2) Let $F_{A}$, $F_{B}$ $\in \tau $. By the definition of $\tau $, $%
c(F_{A}^{c})=F_{A}^{c}$ and $c(F_{B}^{c})=F_{B}^{c}$. By (c3), $c((F_{A}%
\widetilde{\cap }F_{B})^{c})=c(F_{A}^{c}\widetilde{\cup }%
F_{B}^{c})=c(F_{A}^{c})\widetilde{\cup }c(F_{B}^{c})=F_{A}^{c}\widetilde{%
\cup }F_{B}^{c}=(F_{A}\widetilde{\cap }F_{B})^{c}$. So $F_{A}\widetilde{\cap 
}F_{B}\in \tau $.

(T3) Let $\{F_{A_{i}}|i\in J\}\subset \tau $. Since $c$ is order preserving
and $\widetilde{\underset{i\in J}{{\LARGE \cap }}}F_{A_{i}}^{c}\widetilde{%
\subseteq }F_{A_{k}}^{c}$,$\forall k\in J$, then $c(\widetilde{\underset{%
i\in J}{{\LARGE \cap }}}F_{A_{i}}^{c})\widetilde{\subseteq }%
c(F_{A_{k}}^{c})=F_{A_{k}}^{c}$. Then we have $c(\widetilde{\underset{i\in J}%
{{\LARGE \cap }}}F_{A_{i}}^{c})\widetilde{\subseteq }\widetilde{\underset{%
i\in J}{{\LARGE \cap }}}F_{A_{i}}^{c}$. Conversely, by (c2) we have $%
\widetilde{\underset{i\in J}{{\LARGE \cap }}}F_{A_{i}}^{c}\widetilde{%
\subseteq }c(\widetilde{\underset{i\in J}{{\LARGE \cap }}}F_{A_{i}}^{c})$.
Hence, $c((\widetilde{\underset{i\in J}{{\LARGE \cup }}}F_{A_{i}})^{c})=c(%
\widetilde{\underset{i\in J}{{\LARGE \cap }}}F_{A_{i}}^{c})=\widetilde{%
\underset{i\in J}{{\LARGE \cap }}}F_{A_{i}}^{c}=(\widetilde{\underset{i\in J}%
{{\LARGE \cup }}}F_{A_{i}})^{c}$ and $\widetilde{\underset{i\in J}{{\LARGE %
\cup }}}F_{A_{i}}\in \tau $.

Now we will show that with this FP-soft topology $\tau $, $\overline{F_{A}}%
=c(F_{A})$ for every $F_{A}\in FPS(X,E)$. Let $F_{A}\in FPS(X,E)$. Since $(%
\overline{F_{A}})^{c}\in \tau $, then $c(\overline{F_{A}})=\overline{F_{A}}$%
. Since $c$ is order preserving $c(F_{A})\widetilde{\subseteq }c(\overline{%
F_{A}})=\overline{F_{A}}$. Conversely, by (c4) we have $(c(F_{A}))^{c}\in
\tau $. Then since $F_{A}\widetilde{\subseteq }c(F_{A})$ and $\overline{F_{A}%
}$ is the smallest FP-soft closed set over $X$ which contains $F_{A}$, $%
\overline{F_{A}}\widetilde{\subseteq }c(F_{A})$.
\end{proof}

The operator $c$ is called the FP-soft closure operator.

\begin{remark}
By Theorem \ref{kap-oz}(1),(2),(3) and (6) and Theorem \ref{kap-op}, we see
that with a FP-soft closure operator we can associate a FP-soft topology and
conversely with a given FP-soft topology we can associte a FP-soft closure
operator.
\end{remark}

\begin{theorem}
\label{ic}Let $i:FPS(X,E)\rightarrow FPS(X,E)$ be an operator satisfying the
following:

(i1) $i(F_{\widetilde{E}})=F_{\widetilde{E}}$.

(i2) $i(F_{A})\widetilde{\subseteq }F_{A}$, $\forall F_{A}\in FPS(X,E)$.

(i3) $i(F_{A}\widetilde{\cap }F_{B})=i(F_{A})\widetilde{\cap }i(F_{B})$, $%
\forall F_{A},F_{B}\in FPS(X,E)$.

(i4) $i(i(F_{A}))=i(F_{A})$, $\forall F_{A}\in FPS(X,E)$.\newline
Then we can associate a FP-soft topology in the following way:

$\tau =\{F_{A}\in FPS(X,E)|i(F_{A})=F_{A}\}$.\newline
Moreover, with this fuzzy soft topology $\tau $,$(F_{A})^{\circ }$ $%
=i(F_{A}) $ for every $F_{A}\in FPS(X,E)$.
\end{theorem}

\begin{proof}
Similar to that of Theorem \ref{kap-op}.
\end{proof}

The operator $i$ is called the FP-soft interior operator.

\begin{remark}
By Theorem \ref{ic-oz}(1),(2),(3) and (6) and Theorem \ref{ic}, we see that
with a FP-soft interior operator we can associate a FP-soft topology and
conversely with a given FP-soft topology we can associate a FP-soft interior
operator.
\end{remark}

\begin{definition}
Let $(X,{\Large \tau })$ be a FP-soft topological space. A subcollection $%
\mathcal{B}$ of ${\Large \tau }$ is called a base for ${\Large \tau }$ if
every member of ${\Large \tau }$ can be expressed as a union of members of $%
\mathcal{B}$.
\end{definition}

\begin{example}
If we consider the FP-soft topology ${\Large \tau }$ in Example \ref{to},
then one easily see that the family $\mathcal{B=}\{F_{\varnothing
},F_{A_{1}},F_{A_{2}},F_{A_{3}},F_{\widetilde{E}}\}$ is a basis for ${\Large %
\tau }$.
\end{example}

\begin{proposition}
Let $(X,{\Large \tau })$ be a FP-soft topological space and $\mathcal{B}$ is
subfamily of ${\Large \tau }$. $\mathcal{B}$ is a base for ${\Large \tau }$
if and only if for each $e_{\alpha }^{f}$ in $FPS(X,E)$ and for each FP-soft
open Q-nbd $F_{A}$ of $e_{\alpha }^{f}$, there exists a $F_{B}\in \mathcal{B}
$ such that $e_{\alpha }^{f}qF_{B}\widetilde{\subseteq }F_{A}$.
\end{proposition}

\begin{proof}
Let $\mathcal{B}$ be a base for ${\Large \tau }$, $e_{\alpha }^{f}\widetilde{%
\in }FPS(X,E)$ and $F_{A}$ be a FP-soft open Q-nbd of $e_{\alpha }^{f}$.
Then there exists a subfamily $\mathcal{B}^{\prime }$ of $\mathcal{B}$ such
that $F_{A}=\widetilde{{\LARGE \cup }}\left\{ F_{B}\right\vert F_{B}\in 
\mathcal{B}^{\prime }\}$. Suppose that $e_{\alpha }^{f}\overline{q}F_{B}$
for all $F_{B}\in \mathcal{B}^{\prime }$. Then $\alpha +\mu _{B}(e)\leq 1$
and $f(e)\subseteq f_{B}^{c}(e)$ for every $F_{B}\in \mathcal{B}^{\prime }$.
Therefore, we have $\alpha +\mu _{A}(e)\leq 1$ and $f(e)$ is not subset of $%
f_{A}(e)$ since $\mu _{A}(e)=\sup \left\{ \mu _{B}(e)\right\vert F_{B}\in 
\mathcal{B}^{\prime }\}$ and $f_{A}(e)={\LARGE \cup }f_{B}(e)$. This is
contradiction.

Conversely, If $\mathcal{B}$ is not a base for ${\Large \tau }$, then there
exists a $F_{A}\in {\Large \tau }$ such that $F_{C}=\widetilde{{\LARGE \cup }%
}\{F_{B}\in \mathcal{B}:F_{B}\widetilde{\subseteq }F_{A}\}\neq F_{A}$. Since 
$F_{C}\neq F_{A}$, there exists $e\in E$ such that $\mu _{C}(e)<\mu _{A}(e)$
or $f_{C}(e)\subset f_{A}(e)$. Put $\alpha =1-\mu _{C}(e)$ and $%
f(e)=f_{C}^{c}(e)$. Then in both case, we obtain that $e_{\alpha }^{f}qF_{A}$
and $e_{\alpha }^{f}\overline{q}F_{C}$. Therefore, we have $\mu
_{B}(e)+\alpha \leq \mu _{C}(e)+\alpha =1$ and $f(e)\cap f_{B}(e)\subseteq
f(e)\cap f_{C}(e)=\varnothing $; that is, $e_{\alpha }^{f}\overline{q}F_{B}$
for all $F_{B}\in \mathcal{B}$ which contained in $F_{A}$. This is a
contradiction.
\end{proof}

\begin{definition}
Let $(X,{\Large \tau }_{1})$ and $(Y,{\Large \tau }_{2})$ be two FP-soft
topological spaces. A FP-soft mapping $f_{up}:(X,{\Large \tau }%
_{1})\rightarrow (Y,{\Large \tau }_{2})$ is called FP-soft continuous if $%
f_{up}^{-1}(G_{S})\in {\Large \tau }_{1}$, $\forall G_{S}\in {\Large \tau }%
_{2}$.
\end{definition}

\begin{example}
Assume that $X=\{x_{1},x_{2},x_{3}\}$, $Y=\{y_{1},y_{2},y_{3}\}$ are two
universal sets, $E=\{e_{1},e_{2}\}$, $K=\{k_{1},k_{2}\}$ are two parameter
sets and $f_{up}:(X,{\Large \tau }_{1})\rightarrow (Y,{\Large \tau }_{2})$
is a FP-soft mapping, where $u(x_{1})=y_{2}$, $u(x_{2})=y_{1}$, $%
u(x_{3})=y_{3}$ and $p(e_{1})=k_{2}$, $p(e_{2})=k_{1}$. If we take $%
F_{A}=\{((e_{1})_{0,3},\{x_{2},x_{3}\}),((e_{2})_{0,2},\{x_{1},x_{2}\})\}$, $%
G_{S}=\{((k_{1})_{0,2},\{y_{1},y_{2}\}),((k_{2})_{0,3},\{y_{1},y_{3}\})\}$, $%
{\Large \tau }_{1}=$ $\{F_{\varnothing },F_{\widetilde{E}},F_{A}\}$ and $%
{\Large \tau }_{2}=$ $\{\widetilde{0}_{K},\widetilde{1}_{K},G_{S}\}$, then $%
f_{up}$ is a FP-soft continuous mapping.
\end{example}

The constant mapping $f_{up}:(X,{\Large \tau }_{1})\rightarrow (Y,{\Large %
\tau }_{2})$ not continuous in general.

\begin{example}
Assume that $X=\{x_{1},x_{2},x_{3}\}$, $Y=\{y_{1},y_{2}\}$ are two universal
sets, $E=\{e_{1},e_{2}\}$, $K=\{k_{1},k_{2}\}$ are two parameter sets and $%
f_{up}:(X,{\Large \tau }_{1})\rightarrow (Y,{\Large \tau }_{2})$ is a
constant FP-soft mapping, where $u(x_{1})=u(x_{2})=u(x_{3})=y_{2}$ and $%
p(e_{1})=p(e_{2})=k_{1}$. If we take $G_{S}=\{((k_{1})_{0,2},\{y_{1},y_{2}%
\}),((k_{2})_{0,5},\{y_{2},y_{3}\})\}$, ${\Large \tau }_{1}=$ $%
\{F_{\varnothing },F_{\widetilde{E}}\}$ and ${\Large \tau }_{2}=$ $%
\{F_{\varnothing },F_{\widetilde{K}},G_{S}\}$, $f_{up}$ is not a FP-soft
continuous since $f_{up}^{-1}(G_{S})\notin {\Large \tau }_{1}$.
\end{example}

Let $\alpha \in \lbrack 0,1]$. We denote by $\alpha _{E}$ the constant fuzzy
set on $E$, i.e $\mu _{\alpha _{E}}(e)=\alpha $ for all $e\in E$ and $\alpha
\in \lbrack 0,1]$.

\begin{definition}
Let $F_{A}\in FPS(X,E)$. $F_{A}$ is called $\alpha -A$-universal FP-soft set
if $\mu _{A}(e)=\alpha $ and $f_{A}(e)=X$ for all $e\in A$, denoted by $F_{%
\widetilde{\alpha }_{A}}$.
\end{definition}

\begin{definition}
(see \cite{Bh})A FP-soft topology is called enriched if it satisfies $F_{%
\widetilde{\alpha }_{A}}\in \tau $ and $F_{\widetilde{\alpha }_{A}}^{c}\in
\tau $ for all $\alpha \in (0,1]$.
\end{definition}

\begin{theorem}
Let $(X,{\Large \tau }_{1})$ be a enriched FP-soft topological space, $(Y,%
{\Large \tau }_{2})$ be a FP-soft topological space and $f_{up}:FPS(X,E)%
\rightarrow FPS(Y,K)$ be a constant FP-soft mapping. Then $f_{up}$ is
FP-soft continuous.
\end{theorem}

\begin{proof}
Let $G_{S}\in {\Large \tau }_{2}$. Put $f_{up}^{-1}(G_{S})=F_{A}$. Then $%
A=p^{-1}(S)=\alpha _{E}$ where $\alpha =\underset{k\in K}{\sup }\{\mu
_{S}(k)\}$ and \newline
$f_{A}(e)=u^{-1}(g_{S}(p(e))=\left\{ 
\begin{array}{cl}
X & \text{, }g_{S}(p(e))\neq \varnothing \\ 
\varnothing & \text{, otherwise}%
\end{array}%
\right. $ for all $e\in E$. Hence $F_{A}=F_{\widetilde{\alpha }_{E}}\in 
{\Large \tau }_{1}$ or $F_{A}=F_{\widetilde{\alpha }_{E}}^{c}\in {\Large %
\tau }_{1}$ and so $f_{up}:(X,{\Large \tau }_{1})\rightarrow (Y,{\Large \tau 
}_{2})$ is FP-soft continuous.
\end{proof}

\begin{theorem}
Let $(X,{\Large \tau }_{1})$ and $(Y,{\Large \tau }_{2})$ be two FP-soft
topological spaces and $f_{up}:FPS(X,E)\rightarrow FPS(Y,K)$ be a FP-soft
mapping. Then the following are equivalent:

(1) $f_{up}$ is FP-soft continuous;

(2) $f_{up}^{-1}(G_{S})$ is FP-soft closed for every FP- closed set $G_{S}$
over $Y$;

(3) $f_{up}(\overline{F_{A}})\widetilde{\subset }\overline{f_{up}(F_{A})}$, $%
\forall F_{A}\in FPS(X,E)$;

(4) $\overline{f_{up}^{-1}(G_{S})}\widetilde{\subset }f_{up}^{-1}(\overline{%
G_{S}})$, $\forall G_{S}\in FPS(Y,K)$;

(5) $f_{up}^{-1}(G_{S}^{\circ })\widetilde{\subset }(f_{up}^{-1}(G_{S}))^{%
\circ }$, $\forall G_{S}\in FPS(Y,K)$.
\end{theorem}

\begin{proof}
(1)$\Rightarrow $(2) It is obvious from Theorem \ref{fo} (9).

(2)$\Rightarrow $(3) Let $F_{A}\in FPS(X,E)$. Since $F_{A}\widetilde{\subset 
}f_{up}^{-1}(f_{up}(F_{A}))F_{A}\widetilde{\subset }f_{up}^{-1}(\overline{%
f_{up}(F_{A})})\in {\Large \tau }_{1}^{\prime }$. Therefore we have $%
\overline{F_{A}}\widetilde{\subset }f_{up}^{-1}(\overline{f_{up}(F_{A})})$.
By Theorem \ref{fo} (4), we get $f_{up}(\overline{F_{A}})\widetilde{\subset }
$ $f_{up}(f_{up}^{-1}(\overline{f_{up}(F_{A})})\widetilde{\subset }\overline{%
f_{up}(F_{A})}$.

(3)$\Rightarrow $(4) Let $G_{S}\in FPS(Y,K)$. If we choose $%
f_{up}^{-1}(G_{S})$ instead of $F_{A}$ in (3), then $f_{up}(\overline{%
f_{up}^{-1}(G_{S})})\widetilde{\subset }\overline{f_{up}(f_{up}^{-1}(G_{S}))}%
\widetilde{\subset }\overline{G_{S}}$. Hence by Theorem \ref{fo}(3), $%
\overline{f_{up}^{-1}(G_{S})}\widetilde{\subset }f_{up}^{-1}(f_{up}(%
\overline{f_{up}^{-1}(G_{S})}))$ $\widetilde{\subset }f_{up}^{-1}(\overline{%
G_{S}})$.

(4)$\Leftrightarrow $(5) These follow from Theorem \ref{fo} (9) and \
Theorem \ref{ik}.

(5)$\Rightarrow $(1) Let $G_{S}\in {\Large \tau }_{2}$. Since $G_{S}$ is a
FP-soft open set, then $f_{up}^{-1}(G_{S})=f_{up}^{-1}(G_{S}^{\circ })%
\widetilde{\subset }(f_{up}^{-1}(G_{S}))^{\circ }\widetilde{\subset }%
f_{up}^{-1}(G_{S})$. Consequently, $f_{up}^{-1}(G_{S})$ is a FP-soft open
and so $f_{up}$ is FP-soft continuous.
\end{proof}

\begin{theorem}
Let $f_{up}:$ $(X,{\Large \tau }_{1})\rightarrow $ $(Y,{\Large \tau }_{2})$
be a FP-soft mapping and $\mathcal{B}$ be a base for ${\Large \tau }_{2}$.
Then $f_{up}$ is FP-soft continuous if and only if $f_{up}^{-1}(G_{S})\in 
{\Large \tau }_{1}$, $\forall G_{S}\in \mathcal{B}$.
\end{theorem}

\begin{proof}
Straightforward.
\end{proof}

\begin{definition}
A family $\mathcal{C}$ of FP-soft sets is a cover of a FP-soft set $F_{A}$
if $F_{A}\widetilde{\subseteq }\widetilde{\cup }\left\{ F_{A_{i}}\right.
:F_{A_{i}}\in \mathcal{C},i\in J\}$. It is a FP-soft open cover if each
member of $\mathcal{C}$ is a FP-soft open set. A subcover of $\mathcal{C}$
is a subfamily of $\mathcal{C}$ which is also a cover.
\end{definition}

\begin{definition}
A family $\mathcal{C}$ of FP-soft sets has the finite intersection property
if the intersection of the members of each finite subfamily of $\mathcal{C}$
is not empty FP-soft set.
\end{definition}

\begin{definition}
A FP-soft topological space $(X,{\Large \tau })$ is FP-compact if each
FP-soft open cover of $F_{\widetilde{E}}$ has a finite subcover.,
\end{definition}

\begin{example}
Let $X=\{x_{1},x_{2},...\}$, $E=\{e_{1},e_{2},...\}$ and $%
F_{A_{n}}=\{((e_{i})_{\frac{1}{n}},X-\{x_{1},x_{2},...x_{n}\}):i=1,2,...\}$.
Then ${\Large \tau =\{}F_{A_{n}}:n=1,2,...\}\cup \{F_{\varnothing }$,$F_{%
\widetilde{E}}\}$ is a FP-soft topology on $X$ and $(X,{\Large \tau })$ is
FP-compact.
\end{example}

\begin{theorem}
A FP-soft topological space is FP-soft compact if and only if each family of
FP-soft closed sets with the finite intersection property has a non empty
FP-soft intersection.
\end{theorem}

\begin{proof}
If $\mathcal{C}$ is a family of FP-soft sets in a FP-soft topological space $%
(X,{\Large \tau })$, then $\mathcal{C}$ is a cover of $F_{\widetilde{E}}$ if
and only if one of the following conditions holds:

$(1)$ $\widetilde{\cup }\left\{ F_{A_{i}}\right. :F_{A_{i}}\in \mathcal{C}%
,i\in J\}=$ $F_{\widetilde{E}}$

$(2)$ $(\widetilde{\cup }\left\{ F_{A_{i}}\right. :F_{A_{i}}\in \mathcal{C}%
,i\in J\})^{c}=$ $F_{\widetilde{E}}^{c}=F_{\varnothing }$

$(3)$ $\widetilde{\cap }\left\{ F_{A_{i}}^{c}\right. :F_{A_{i}}\in \mathcal{C%
},i\in J\}=F_{\varnothing }$ \newline
Hence the FP-soft topological space is FP-soft compact if and only if each
family of FP-soft open sets over $X$ such that no finite subfamily covers $%
F_{\widetilde{E}}$, fails to be a cover, and this is true if and only if
each family of FP-soft closed sets which has the finite intersection
property has a nonempty FP-soft intersection.
\end{proof}

\begin{theorem}
Let $(X,{\Large \tau }_{1})$ and $(Y,{\Large \tau }_{2})$ be FP-soft
topological spaces and $f_{up}:FPS(X,E)\rightarrow FPS(Y,K)$ be a FP-soft
mapping. If $(X,{\Large \tau }_{1})$ is FP-soft compact and $f_{up}$ is
FP-soft continuous surjection, then $(Y,{\Large \tau }_{2})$ is FP-soft
compact.
\end{theorem}

\begin{proof}
Let $\mathcal{C}=\{G_{S_{i}}:i\in J\}$ be a cover of $G_{\widetilde{K}}$ by
FP-soft open sets. Then since $f_{up}$ is FP-soft continuous, the family of
all FP-soft sets of the form $f_{up}^{-1}(G_{S})$, for $G_{S}\in \mathcal{C}$%
, is a FP-soft open cover of $F_{\widetilde{E}}$ which has a finite
subcover. However, since $f_{up}$ is surjective, then $%
f_{up}(f_{up}^{-1}(G_{S})=G_{S}$ for any FP-soft set $G_{S}$ over $Y$. Thus,
the family of images of members of the subcover is a finite subfamily of $%
\mathcal{C}$ which covers $G_{\widetilde{K}}$. Consequently, $(Y,{\Large %
\tau }_{2})$ is FP-soft compact.
\end{proof}

\textbf{Conclusion.} Topology is a branch of mathematics, whose concepts
exist not only in almost all branches of mathematics, but also in many real
life applications. In this paper, we introduce the topological structure of
fuzzy parametrized soft sets and fuzzy parametrized soft mappings. We study
some fundemental concepts in fuzzy parametrized soft topological spaces such
as closures, interiors, bases, compactness and continuity. Some basic
properties of these concepts are also presented. This paper will form the
basis for further applications of topology.


\begin{thebibliography}{99}
\bibitem{Ak} Acar, U., Koyuncu, F. and Tanay, B. \textit{Soft sets and soft
rings}, Comput. Math. Appl. \textbf{59}, 3458-3463, 2010.

\bibitem{AC} Akta\c{s}, H. and \c{C}a\u{g}man, N. \textit{Soft sets and soft
groups}, Inform. Sci. \textbf{177} (13), 2726-2735, 2007.

\bibitem{Az} Atmaca, S. and Zorlutuna, \.{I}. \textit{On fuzzy soft
topological spaces}, Ann. Fuzzy Math. Inform. \textbf{5} (2), 377-386, 2013.

\bibitem{AA} Ayg\"{u}no\u{g}lu, A. and Ayg\"{u}n, H. \textit{Introduction to
fuzzy soft groups}, Comput. Math. Appl. \textbf{58,} 1279-1286, 2009.

\bibitem{Bh} Varol, B. P. and Ayg\"{u}n, H. \textit{Fuzzy soft topology},
Hacet. J. Math. Stat. \textbf{41} (3), 407-419, 2012.

\bibitem{ac1} \c{C}a\u{g}man, N., Karata\c{s}, S. and Enginoglu, S. \textit{%
Soft topology}, Comput. Math. Appl. \textbf{62,} 351-358, 2011.

\bibitem{ce2} \c{C}a\u{g}man, N., \c{C}\i tak, F. and Enginoglu, S. \textit{%
FP-soft set theory and its applications}, Ann. Fuzzy Math. Inform. \textbf{2}
(2), 219-226, 2011.

\bibitem{Cd} \c{C}a\u{g}man, N. and Deli, \.{I}. \textit{Products of FP-soft
sets and their applications}, Hacet. J. Math. Stat. \textbf{41} (3),
365-374, 2012.

\bibitem{Cd1} \c{C}a\u{g}man, N. and Deli, \.{I}. Means \textit{of FP-soft
sets and their applications}, Hacet. J. Math. Stat. \textbf{41} (5),
615-625, 2012.

\bibitem{CE} \c{C}a\u{g}man, N. and Enginoglu, S. \textit{Soft set theory
and uni-int decision making}, Eur. J. Oper. Res. \textbf{207}, 848-855, 2010.

\bibitem{ce1} \c{C}a\u{g}man, N. and Enginoglu, S. \textit{Soft matrix
theory and its decision making}, Comput. Math. Appl. \textbf{59}, 3308-3314,
2010.

\bibitem{yc} \c{C}elik, Y., Ekiz, C. and Yamak, S. \textit{A new view on
soft rings}, Hacet. J. Math. Stat. \textbf{40 }(2), 273-286, 2011.

\bibitem{ch} Chang, C. L. \textit{Fuzzy topological spaces}, J. Math. Anal.
Appl. \textbf{24}, 182-190, 1968.

\bibitem{dc} Chen, D., Tsang, E. C. C., Yeung, D. S. and Wang, X. \textit{%
The parametrization reduction of soft sets and its applications}, Comput.
Math. Appl. \textbf{49}, 757-763, 2005.

\bibitem{chen} B. Chen, \textit{Soft semi-open sets and related properties
in soft topological spaces}, Appl. Math. Inf. Sci., 7 (1) (2013), 287-294.

\bibitem{chen1} B. Chen, \textit{Some local properties of soft semi-open
sets, }Discrete Dyn. Nat. Soc., Article ID 298032, 6 pages, 2013.

\bibitem{F} Feng, F., Jun, Y.B. and Zhao, X. \textit{Soft semirings},
Comput. Math. Appl. \textbf{56,} 2621-2628, 2008.

\bibitem{fj} Feng, F., Jun, Y.B., Liu, X. and Li, L.F. \textit{An adjustable
approach to fuzzy soft set based decision making}, J. Comput. Appl. Math. 
\textbf{234}, 10-20, 2010.

\bibitem{geor} Georgiou, D. N., Megaritis, A.C. and Petropoulos, V.I. 
\textit{On soft topological spaces}, Appl. Math. Inf. Sci. (2013), in press.

\bibitem{io} Inan, E. and \"{O}zt\"{u}rk,\textit{\ }M. A\textit{. Fuzzy soft
rings and fuzzy soft ideals}, Neural Comput. and Applic DOI
10.1007/s00521-011-0550-5.

\bibitem{yb} Jun, Y. B. and Park, C.H. \textit{Applications of soft sets in
ideal theory of BCK/BCI-algebras}, Inform. Sci. \textbf{178}, 2466-2475,
2008.

\bibitem{zk} Kong, Z., Gao, L.Q. and Wang, L.F. \textit{Comment on a fuzzy
soft set theoretic approach to decision making problems}, J. Comput. Appl.
Math. \textbf{223}, 540-542, 2009.

\bibitem{MB2} Maji, P. K., Roy, A. R. and Biswas, R. \textit{An application
of soft sets in desicion making problem}, Comput. Math. Appl. \textbf{44},
1077-1083, 2002.

\bibitem{MB3} Maji, P. K., Biswas, R. and Roy, A. R. \textit{Fuzzy soft sets}%
, Journal of Fuzzy Mathematics, \textbf{203} (2), 589-602, 2001.

\bibitem{PM} Ming, P. B. and Ming, L. Y. \textit{Fuzzy topology
I.Neighbourhood structure of a fuzzy point and Moore-Smith convergence}, J.
Math. Anal. Appl. \textbf{76,} 571-599, 1980.

\bibitem{M} Molodtsov, D. \textit{Soft set theory-First results}, Comput.
Math. Appl. \textbf{37} (4/5), 19-31, 1999.

\bibitem{sn} Shabir, M. and Naz, M. \textit{On soft topological spaces},
Comput. Math. Appl. \textbf{61}, 1786-1799, 2011.

\bibitem{arr} Roy, A.R. and Maji, P.K. \textit{A fuzzy soft set theoretic
approach to decision making problems}, J. Comput. Appl. Math. \textbf{203},
412-418, 2007.

\bibitem{TK} Tanay, B. and Kandemir, M. B. \textit{Topological structures of
fuzzy soft sets}, Comput. Math. Appl. \textbf{61}, 412-418, 2011.

\bibitem{Ts} \c{S}imsekler, T. and Y\"{u}ksel, S. \textit{Fuzzy soft
topological spaces}, Ann. Fuzzy Math. Inform. \textbf{5} (1), 87-96, 2013.

\bibitem{x2} Xiao, Z., Gong, K. and Zou, Y. \textit{A combined forecasting
approach based on fuzzy soft sets}, J. Comput. Appl. Math. \textbf{228},
326-333, 2009.

\bibitem{Z} Zadeh, L. A. \textit{Fuzzy sets}, Inform. and Control \textbf{8}%
, 338-353, 1965.

\bibitem{zo} Zorlutuna, \.{I}., Akda\u{g}, M., Min, W. K. and Atmaca, S. 
\textit{Remarks on soft topological spaces}, Ann. Fuzzy Math. Inform. 
\textbf{3} (2), 171-185, 2011.
\end{thebibliography}
\end{document}